\documentclass[a4paper,11pt,twoside]{amsart}
\usepackage[english]{babel}
\usepackage[utf8]{inputenc}

\usepackage[margin=2.5cm]{geometry}
\usepackage{enumitem}

\usepackage{bold-extra}
\usepackage{ mathrsfs }

\usepackage{graphics}
\usepackage{aliascnt}
\usepackage[pdftex,citecolor=green,linkcolor=red]{hyperref}

\usepackage{amsmath}
\usepackage{amsfonts}
\usepackage{amssymb}
\usepackage{amsthm}
\usepackage{mathtools}

\newtheorem{theorem}{Theorem}[section]
\newtheorem{corollary}[theorem]{Corollary}

\newtheorem{lemma}[theorem]{Lemma}

\newtheorem*{proposition*}{Proposition}

\theoremstyle{definition}

\newtheorem{remark}[theorem]{Remark}

\numberwithin{equation}{section}
\newcommand\eps{\varepsilon}

\newcommand\N{\mathbb{N}}

\DeclareMathOperator\rad{rad}

\renewcommand{\leq}{\leqslant}
\renewcommand{\le}{\leqslant}

\renewcommand{\ge}{\geqslant}

\usepackage[usenames,dvipsnames]{color}

\makeatletter
\@namedef{subjclassname@2020}{\textup{2020} Mathematics Subject Classification}
\makeatother

\begin{document}

\author[J.D. Lichtman]{Jared Duker Lichtman}
\address{
Department of Mathematics\\ 
Stanford University\\ 
450 Jane Stanford Way\\
Stanford, CA 94305-2125\\
USA}
\email{j.d.lichtman@stanford.edu}

\title[The $abc$ conjecture almost always]{The $abc$ conjecture is true almost always}

\begin{abstract}
Let $\rad(n)$ denote the product of distinct prime factors of an integer $n\ge1$. The celebrated $abc$ conjecture asks whether every solution to the equation $a+b=c$ in triples of coprime integers $(a,b,c)$ must satisfy $\rad(abc) > K_\eps\, c^{1-\eps}$, for some constant $K_\eps>0$. In this expository note, we present a classical estimate of de Bruijn that implies almost all such triples satisfy the $abc$ conjecture, in a precise quantitative sense. Namely, there are at most $O(N^{2/3})$ many triples of coprime integers in a cube $(a,b,c)\in\{1,\ldots,N\}^3$ satisfying $a+b=c$ and $\rad(abc) < c^{1-\eps}$. The proof is elementary and essentially self-contained.

Beyond revisiting a classical argument for its own sake, this exposition is aimed to contextualize a new result of Browning, Lichtman, and Ter\"av\"ainen, who prove a refined estimate $O(N^{33/50})$, giving the first power-savings since 1962.
\end{abstract}

\maketitle

\setcounter{tocdepth}{1}

\section{Introduction}

In this note, we are concerned with solutions to the simple equation $a+b=c$ in triples of integers $(a,b,c)\in \N^3$. First, if $a$ and $b$ are both divisible by a prime $p$, then $c=a+b$ is also divisible by $p$. So we may divide by $p$ to give $\frac{a}{p}+\frac{b}{p}=\frac{c}{p}$. We may continue to divide by shared primes until $a$ and $b$ are coprime integers. Then $c=a+b$ is also coprime to both $a$ and $b$, which we assume hereafter.

For any integer $n\ge1$, we may factor $n=p_1^{a_1}\cdots p_k^{a_1}$ for a unique collection of primes $p_1<\cdots < p_k$ and exponents $a_1,\ldots, a_k\ge1$. Then let $\rad(n)=p_1\cdots p_k$ denote the product of distinct prime factors of $n$. That is, we let
\begin{align*}
\rad(n)=\prod_{p\mid n}p.    
\end{align*}
The celebrated $abc$ conjecture of Masser and Oesterl\'e asserts that for any $\eps>0$ there is a constant $K_\eps>0$ such that every triple $(a,b,c)\in\N^3$ of coprime integers solving the equation $a+b=c$ must also satisfy $\rad(abc) \ > \ K_\eps\, c^{1-\eps}$. Equivalently, only finitely many such solutions to $a+b=c$ may satisfy
\begin{align}\label{eq:abc}
\gcd(a,b)=1\qquad{\rm and}\qquad \rad(abc) < c^{1-\eps}.
\end{align}

The $abc$ conjecture is intrinsically attractive for suggesting that basic addition (from the simple equation $a+b=c$) has a subtle connection to multiplication (from the prime factors of $abc$). Moreover, the $abc$ conjecture is highly coveted for its potential to generalize many celebrated results in number theory, including Fermat's last theorem. See \cite{GT} for a nice exposition on these connections.

The best unconditional progress towards \eqref{eq:abc} is due to Stewart and Yu~\cite{styu}, who showed that only finitely many coprime solutions to $a+b=c$ may satisfy $\rad(abc) < (\log c)^{3-\eps}$. This condition is exponentially stricter than \eqref{eq:abc}. Recently, Pasten~\cite{hector} obtained a new, subexponentially strict condition of $\rad(abc) < \exp((\log\log c)^{2-\eps})$, in the special case when $a<c^{1-\eps}$.

If we count solutions to $a+b=c$ in a cube $(a,b,c)\in\{1,\ldots,N\}^3$, there are $O(N^2)$ many such triples. In fact there are about $\frac{6}{\pi^2} N^2$ such coprime triples.
We may then consider the set of exceptional triples,
\begin{align}
\mathcal E(N) \; = \; \bigg\{(a,b,c)\in \{1,\ldots,N\}^3 \;:\; \gcd(a,b)=1, %\quad
\begin{array}{c}
 a+b=c\\
\rad(abc) \, < \, c^{1-\eps}
\end{array}
\bigg\}.
\end{align}
The $abc$ conjecture is equivalent to the bound $|\mathcal E(N)| \le O_\eps(1)$ for all $N\ge1$. 

In this note, we show that among the $O(N^2)$ many coprime solutions to $a+b=c$ in a cube of length $N$, at most $O(N^{2/3})$ of those solutions also satisfy \eqref{eq:abc}. This makes precise the statement that the $abc$ conjecture is true almost always.

\begin{theorem} \label{thm:deBr}
We have $|\mathcal E(N)| \le O(N^{2/3})$.
\end{theorem}

The proof of Theorem \ref{thm:deBr} is elementary and essentially self-contained. The main argument is a half-page (see p. 4), assuming a classical estimate on radicals, which itself has a one page. Namely, Lemma \ref{lem:radical} (see p. 3) shows that very few integers $n\le N$ can share a common radical $r=\rad(n)$. To the author's knowledge, Theorem \ref{thm:deBr} has not appeared in the literature (prior to \cite{br}), though it is folklore to experts.

Beyond presenting this elegant argument for its own sake, our exposition is aimed to motivate a 2024 result of Browning, Lichtman, and Ter\"av\"ainen \cite{br}.

\begin{theorem}[\cite{br}] \label{thm:power}
We have $|\mathcal E(N)| \le O(N^{33/50})$.
\end{theorem}

Theorem \ref{thm:power} refines Theorem \ref{thm:deBr}, reducing the exponent $\frac{2}{3}=0.666\cdots$ down to $\frac{33}{50}=0.66$.
However, unlike the half-page elementary proof of Theorem \ref{thm:deBr}, the proof of Theorem \ref{thm:power} is over 15 pages, and itself relies on a collection of more sophisticated techniques across the literature---namely, Fourier analysis, Thue equations, geometry of numbers, and determinant methods (see \cite{bl} for an accessible introduction). These techniques are used to count integer points on varieties. We hope this note serves as a friendly entry-point, and motivation for further reading on these subjects.

\section*{Notation}

Here $f(n) = O(g(n))$ means that $|f(n)| \le C\,g(n)$ for all $n>0$ where $C>0$ is an absolute constant; while $f(n) = O_\eps(g(n))$ means $|f(n)| \le C\,g(n)$ for some $C=C_\eps>0$ depending on $\eps$. $f(n) = o(1)$ means that $\lim_{n\to\infty}f(n) = 0$. Finally, $\tau(n)$ denotes the number of divisors of an integer $n\ge1$.

\section{Divisor bound and radical estimate}

We begin with the classical divisor bound, which is one of the most useful and versatile estimates in number theory. We provide its proof for completeness in Lemma \ref{lem:div} below. Moreover, its proof will foreshadow that of Lemma \ref{lem:radical}, for the main radical estimate.

\begin{lemma}[Divisor bound] \label{lem:div}
We have $\tau(n) \le n^{o(1)}$ for any integer $n\ge1$.
\end{lemma}
\begin{proof}
We factor $n = p_1^{a_1}\cdots p_k^{a_k}$ into primes, so that $\tau(n) = (a_1+1)\cdots(a_k+1)$. So for any $0<\eps<1$, we have
\begin{align}\label{eq:taunn}
\frac{\tau(n)}{n^\eps} = \prod_{i\le k}\frac{a_i+1}{p_i^{a_i \eps}}.
\end{align}
If $p_i^{\eps}\ge 2$ then $\frac{a_i+1}{p_i^{a_i \eps}}\le \frac{a_i+1}{2^{a_i}}\le 1$. Otherwise $p_i<2^{1/\eps}$, in which case
\begin{align*}
\frac{a_i+1}{p_i^{a_i \eps}} \le \frac{a_i+1}{2^{a_i \eps}} \le \frac{a_i+1}{a_i \eps/2+1} \le \frac{2}{\eps},
\end{align*}
using $p_i\ge2$ and $2^x \ge x/2+1$ for all $x\ge0$. Thus \eqref{eq:taunn} becomes
\begin{align*}
\frac{\tau(n)}{n^\eps} \le \prod_{p<2^{1/\eps}}\frac{2}{\eps} \le (2/\eps)^{2^{1/\eps}}.
\end{align*}
In particular $\tau(n) \le O_\eps(n^\eps)$. Taking $\eps\to 0$ implies $\tau(n)\le n^{o(1)}$ as desired.
\end{proof}

The following radical estimate shows that there can be at most $N^{o(1)}$ many integers $n\le N$ with a common radical $r=\rad(n)$.

\begin{lemma} \label{lem:radical}
Let $r\le N$ and let $\mathcal{R}=\mathcal{R}(r,N)= \{ n\le N : \rad(n)=r \}$. Then we have
\[
|\mathcal{R}| \le N^{o(1)}.
\]
\end{lemma}

\begin{proof}
We factor $r = p_1 \cdots p_k$
into distinct primes $p_1<\cdots<p_k$. Then the integers $n$ with $\rad(n)=r$
are of the form $n = p_1^{a_1}p_2^{a_2}\cdots p_k^{a_k}$ for any exponents $a_1,\dots,a_k\ge 1$. That is,
\[
\mathcal{R}:= \Bigl\{ n\le N : \rad(n)=r \Bigr\} \subset \Bigl\{\, p_1^{a_1}\cdots p_k^{a_k} : a_1,\dots,a_k\ge1\,\Bigr\}.
\]

For any $\eps>0$, we have $|\mathcal R| \le \sum_{n\in \mathcal{R}}(N/n)^\eps$. Then we may factor the Dirichlet series as an Euler product,
\begin{align}\label{eq:Eulerprod}
\sum_{n\in \mathcal{R}}\frac{1}{n^{\eps}}
\le \sum_{a_1\ge1}\cdots\sum_{a_k\ge1} \frac{1}{(p_1^{a_1}\cdots p_k^{a_k})^\eps}
=\prod_{i\le k} \Bigl(\sum_{a_i\ge1} \frac{1}{p_i^{\eps a_i}}\Bigr)
=\prod_{i\le k} \frac{1}{p_i^\eps-1}.
\end{align}

If $p_i^{\eps}\ge 2$ then $\frac{1}{p_i^{\eps}-1}\le 1$. Otherwise $p_i<2^{1/\eps}$, in which case
\begin{align*}
\frac{1}{p_i^{\eps}-1} \le \frac{1}{2^{\eps}-1} \le \frac{2}{\eps},
\end{align*}
using $p_i\ge2$ and $2^x \ge x/2+1$ for all $x\ge0$. Thus \eqref{eq:Eulerprod} becomes
\begin{align*}
\frac{|\mathcal{R}|}{N^\eps} \le \sum_{n\in \mathcal{R}}\frac{1}{n^{\eps}} \le \prod_{p<2^{1/\eps}}\frac{2}{\eps} \le (2/\eps)^{2^{1/\eps}}.
\end{align*}
In particular $|\mathcal{R}| \le O_\eps(N^\eps)$. Taking $\eps\to 0$ implies $|\mathcal{R}|\le N^{o(1)}$ as desired. 
\end{proof}

As a direct consequence of Lemma \ref{lem:radical} we deduce:

\begin{corollary}\label{cor:rad}
For any $0\le\lambda\le 1$, we have
\begin{align*}
\big|\big\{n\le N : \rad(n) \le N^{\lambda}\big\}\big| \ = \ N^{\lambda+o(1)}.
\end{align*}
\end{corollary}
\begin{proof}
We split into all possible radicals $r\le N^\lambda$, giving
\begin{align*}
\big|\big\{n\le N : \rad(n) \le N^{\lambda}\big\}\big| &= \sum_{r\le N^{\lambda}}|\mathcal R(r,N)| %\\
%&
\le \sum_{r\le N^{\lambda}}N^{o(1)}= N^{\lambda+o(1)},
\end{align*}
by Lemma \ref{lem:radical}. This completes the proof.
\end{proof}
\begin{remark}
We may optimize Lemma \ref{lem:radical} to show $|\mathcal R(r,N)| \le N^{O(1/\log\log N)}$, taking $\eps = O(1/\log\log N)$.
Similarly optimizing Lemma \ref{lem:div} gives $\tau(n) \le n^{O(1/\log\log n)}$.
\end{remark}

In 1962, de Bruijn \cite{debruijn} obtained the above bound $N^{\lambda+o(1)}$ in the sharper quantitative form $N^{\lambda + O(1/\sqrt{\log N\log\log N})}$. Interestingly, Robert, Stewart and Tenenbaum~\cite{tenenbaum} have proposed a refined $abc$ conjecture, expressing the inequality $\rad(abc) < c^{1-\eps}$ from \eqref{eq:abc} in the quantitative form $\rad(abc) < c^{1-O(1/\sqrt{\log c\log\log c})}$.

\section{Proof of Theorem \ref{thm:deBr}}
Take an exceptional triple $(a,b,c)\in \mathcal E(N)$. That is, $a,b,c$ are coprime integers satisfying $a+b=c$ and $\rad(abc)< c^{1-\eps}$.
Observe that
\begin{align*}
\rad(abc)=\prod_{p\mid abc}p = \Big(\prod_{p\mid a}p\Big)\Big(\prod_{p\mid b}p\Big)\Big(\prod_{p\mid c}p\Big)
=\rad(a)\rad(b)\rad(c),
\end{align*}
since $\rad(n)$ is a multiplicative function. In particular,
\begin{align*}
\rad(ab)\rad(ac)\rad(bc) = \rad(abc)^2 < c^{2-2\eps}.
\end{align*}
Thus there is some choice of $xy\in\{ab,ac,bc\}$ that satisfies
\begin{align*}
\rad(xy)\le c^{\frac23-\eps}.
\end{align*}
So we have
\begin{align*}
|\mathcal E(N)| = \underset{\substack{1\le a,b,c\le N\\ \gcd(a,b)=1\\ \rad(abc) < c^{1-\eps} \\ a+b=c}}{\sum\sum\sum} 1
\ \le \ 3\sum_{r\le N^{2/3-\eps}}\underset{\substack{1\le x,y\le N\\ \gcd(x,y)=1\\r=\rad(xy)}}{\sum\sum} 1 
\ \le \ 3\sum_{r\le N^{2/3-\eps}}\sum_{\substack{1\le n\le N^2\\r=\rad(n)}} \tau(n).
\end{align*}
The divisor bound in Lemma \ref{lem:div} gives $\tau(n)\le N^{o(1)}$, and so 
\begin{align}
|\mathcal E(N)| \ \le \ \sum_{r\le N^{2/3-\eps}}
      \sum_{\substack{n\le N^{2}\\ \rad(n)=r}}N^{o(1)}\le \sum_{r\le N^{2/3-\eps}}N^{o(1)}
      \le N^{\frac23-\eps+o(1)}
\end{align}
by Corollary \ref{cor:rad}. Hence $|\mathcal E(N)| \le O(N^{\frac{2}{3}})$.
This completes the proof of Theorem \ref{thm:deBr}.

\end{document}